\documentclass[12pt]{article}
\usepackage[margin=25mm]{geometry}
\usepackage{amsmath}
\usepackage{amsfonts}
\usepackage{amssymb}
\usepackage{amsthm}
\usepackage{graphicx}
\usepackage{verbatim}
\immediate\write18{texcount -tex -sum  \jobname.tex > \jobname.wordcount.tex}

\newtheorem{theorem}{Theorem}
\newtheorem{lemma}{Lemma}
\newtheorem{conjecture}{Conjecture}
\providecommand{\keywords}[1]
{
  \small	
  \textbf{\textit{Keywords---}} #1
}

\title{Independent sets versus 4-dominating sets \\ in outerplanar graphs}
\author{Dmitrii Taletskii$^{1,2,3}$\\
\small
$^1$ National Research University Higher School of Economics, \\
\small
Bolshaya Pechyorskaya ul. 25/12, Nizhny Novgorod, 603155 Russia\\%
\small
    $^2$Lobachevsky Nizhny Novgorod State University, \\ 
    \small
    ul. Gagarina 23, Nizhny Novgorod, 603950 Russia\\
    \small
    $^3$Saint Petersburg University, 
    7/9 Universitetskaya nab., \\ \small 
    St. Petersburg, 199034 Russia\\%
    \small
    e-mail: dmitalmail@gmail.com
}
\date{} 

\begin{document}
\maketitle

\begin{abstract}
We show that the number of independent sets in every outerplanar graph is greater than the number of its 4-dominating sets.
\end{abstract} \hspace{10pt}


\section{Introduction}
Throughout this paper we consider only simple and finite graphs. Let $G$ be a graph with a vertex set $V(G)$ and an edge set $E(G)$. An \textit{independent set} is a pairwise nonadjacent subset of $V(G)$. For $k \geq 1$, a \textit{$k$-dominating set} $D_k$ is a subset of $V(G)$ such that every vertex not from $D_k$  is adjacent to at least $k$ vertices from $D_k$. We use the abbriveations `IS' and `$k$-DS' for the terms `independent set' and `$k$-dominating set' respectively. Let $i(G)$ (resp. $\partial_k(G)$) be the number of all IS (resp. $k$-DS) of a graph $G$.  Denote by $\mathfrak{D}_k(G)$ the family of all $k$-DS of a graph $G$.

 A planar graph is called \textit{outerplanar} if it has an embedding in the plane such that all vertices belong to the boundary of its outer face. An outerplanar graph $G$ is called \textit{maximal outerplanar} (or \textit{MOP}) if $G + uv$ is not outerplanar for any two non-adjacent vertices $u,v \in V(G)$. It is well-known that every inner face of a MOP is a triangle and every MOP with at least 3 vertices has at least 2 vertices of degree 2. Denote by $\mathcal{OP}$ and $\mathcal{MOP}$ the classes of all outerplanar and maximal outerplanar graphs, respectively.

The definition of $k$-dominating set implies that for every graph $G$, any two non-adjacent vertices $u,v \in V(G)$ and any integer $k \geq 1$ we have $\partial_k(G) \leq \partial_k(G + uv)$. Therefore, a complete graph $K_n$ has the maximum number of $k$-independent sets among all $n$-vertex graphs. Trees with extremal numbers of $k$-dominating sets were described in~\cite{BS06} for $k = 1$ and in~\cite{T22} for $k \geq 2$. In~\cite{N15} and~\cite{GKMP17} for every $k \geq 2$ new upper bounds for the number of $k$-dominating independent sets in $n$-vertex graphs were presented. 

The empty graph $nK_1$ has the maximum possible number of independent sets among all $n$-vertex graphs. Moon and Moser~\cite{MM65} described $n$-vertex graphs with maximal possible number of \textit{maximal independent} (i.e. independent 1-dominating) sets. In \cite{A98} for all $n \geq 5$ the graphs $H_n \in \mathcal{MOP}$ and $H'_n \in \mathcal{MOP}$ with the maximum and minimum  possible number of IS among all $n$-vertex MOPs were described.

\begin{figure}[h!]

\begin{center}
\includegraphics[scale=0.35
]{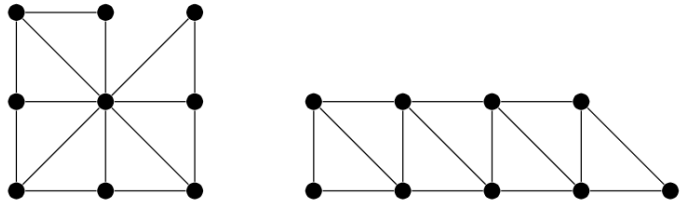}
\caption{Graphs $H'_9$ and $H_9$}
\end{center}
\end{figure}

The main result of this paper is the following fact:

\begin{theorem}\label{thm1}
For every outerplanar graph $G$ we have $i(G) > \partial_4(G)$.
\end{theorem}


The rest of the paper is organized as follows. In Section 2 we introduce some graph terminology and present a maximal outerplanar graph partition. In Section 3 we prove Theorem 1. In Section 4 we consider a possible generalization of~Theorem~\ref{thm1} and obtain a similar result for the class of trees.

\section{Preliminaries}

\subsection{Basic terminology}

A \textit{tree} is a connected acyclic graph and a \textit{leaf} is a vertex of degree one in a tree. A \textit{support vertex} in a tree is a vertex which is adjacent to at least one leaf. A \textit{diameter} of a connected graph is the maximum possible distance between its vertices. A simple path is called \textit{diametral}, if its length is equal to the diameter of a graph. Clearly, the ends of any diametral path in a tree are leaves. 

For a graph $G \in \mathcal{OP}$ we consider \textit{a weak dual graph} $T(G)$, such that the vertices of $T(G)$ correspond to the inner faces of $G$ and two vertices are adjacent if and only if the corresponding faces have a common edge. It is well-known that if $G \in \mathcal{MOP}$, then $T(G)$ is a subcubic tree.

Let $G$ be a graph and $U \subset V(G)$ be its vertex subset. Let $G \setminus U$ be an induced subgraph of $G$ with the set of vertices $V(G) \setminus U$.

An inner face of a MOP is called an \textit{end face} if it has a vertex of degree 2. Note that a face $f$ of a MOP $G$ is an end face if and only if the corresponding vertex in $T(G)$ has degree at most one. We say that a graph $G$ \textit{contains} a face $f$, if all vertices and edges from $f$ belong to $G$.

Suppose that $f$ and $f'$ are two adjacent inner faces of $G$ with a common edge $uv$. Denote by $G[f;f']$ the maximal by inclusion subgraph of $G$ such that it contains $f$, does not contain $f'$ and $uv$ is its outer edge. It is easy to see that if $G \in \mathcal{MOP}$, then $G[f;f'] \in \mathcal{MOP}$.

\subsection{Independent and 4-dominating sets}

Let $G$ be a graph and $u,v \in V(G)$. Let $i(G,u^+)$ (resp. $i(G,u^-)$) be the number of IS of $G$ which contain (resp. do not contain) $u$. We denote by $i(G,u^+,v^+)$ the number of IS of $G$ such that $u,v \in I$ and denote the values $i(G,u^+,v^-)$, $i(G,u^-,v^+)$ and $i(G,u^-,v^-)$ in the similar way. Clearly, for any vertices $u,v \in V(G)$ we have
$$i(G) = i(G,u^+,v^+) + i(G,u^+,v^-) + i(G,u^-,v^+) + i(G,u^-,v^-).$$
Moreover, $i(G,u^+,v^+) = 0$, if and only if $uv \in E(G)$.

Let $\partial_4(G,u^+)$ (resp. $\partial_4(G,u^-)$) be the number of 4-DS in $G$ which contain (resp. do not contain) $u$. We denote the values $\partial_4(G,u^+,v^+)$, $\partial_4(G,u^+,v^-)$, $\partial_4(G,u^-,v^+)$ and $\partial_4(G,u^-,v^-)$ in the similar way. Again, for any vertices $u,v \in V(G)$ we have
$$\partial_4(G) = \partial_4(G,u^+,v^+) + \partial_4(G,u^+,v^-) + \partial_4(G,u^-,v^+) + \partial_4(G,u^-,v^-).$$

Note that $\partial_4(G,u^+,v^+) > 0$ for any graph $G$. It is easy to see that $\partial_4(G,u^-,v^+) > 0$, if and only if $deg(u) \geq 4$. Moreover, $\partial_4(G,u^-,v^-) > 0$, if and only if either $\min(deg(u),deg(v)) \geq 5$ or $uv \notin E(G)$ and $\min(deg(u),deg(v)) \geq 4$.

\subsection{MOP-partition}

Consider a graph $G \in \mathcal{MOP}$ and its edge $uv \in E(G)$. Let $G_L,G_R \in \mathcal{MOP}$ be two maximal by inclusion subgraphs of $G$ such that $ E(G_L) \cap E(G_R) = uv$ and $uv$ is an outer edge in both $G_L$ and $G_R$. Clearly, if $uv$ is not an outer edge in $G$ than both $G_L$ and $G_R$ have at least 3 vertices (otherwise one of them has 2 vertices and the other coincides with $G$). Call a triple $(G_L,G_R,uv)$ a \textit{MOP-partition} of $G$. For the given MOP-partition $(G_L,G_R,uv)$ we shall use the notation
$$\mathcal{I}_{00} = i(G_R,u^-,v^-), \ \mathcal{I}_{01} = i(G_R,u^-,v^+), \ \mathcal{I}_{10} = i(G_R,u^+,v^-).$$
Since $i(G_R,u^+,v^+) = 0$, we have $i(G_R) = \mathcal{I}_{00} + \mathcal{I}_{01} + \mathcal{I}_{10}$. Morever, if $G_R$ has at least 3 vertices, than $\mathcal{I}_{00} > \textrm{max}(\mathcal{I}_{01},\mathcal{I}_{10})$.

We now introduce a similar concept for 4-DS. For the given MOP-partition $(G_L,G_R,uv)$ of a graph $G$ let $\mathcal{D}_{11} = \partial_4(G_R,u^+,v^+)$. Denote by $\mathcal{D}^k_{01}$ the number of 4-DS $D'$ of a graph $G_R \setminus u$ such that $v \in D'$ and the vertex $u$ has at least $\max(0,4-k)$ neighbors from $D' \setminus v$ in $G$. In other words, $\mathcal{D}^k_{01}$ is the number of `almost 4-dominating' sets $D'$ of $G_R$ with respect to the vertex $u$. Denote $\mathcal{D}^k_{10}$ in the same way as $\mathcal{D}^k_{01}$. Finally, let $\mathcal{D}^{kl}_{00}$ be the number of 4-DS $D''$ of a graph $G_R \setminus \{u,v\}$, such that in the graph $G_R$ the vertices $u$ and $v$ have at least $\max(0,4-k)$ and $\max(0,4-l)$ neighbors from $D''$, respectively.

From the definitions of $\mathcal{D}^k_{01}$, $\mathcal{D}^k_{10}$ and $\mathcal{D}^{kl}_{00}$ it follows immediately that for all $1 \leq k' \leq k$ and $1 \leq l' \leq l$ we have $\mathcal{D}^{k'}_{01} \leq \mathcal{D}^k_{01}, \mathcal{D}^{l'}_{10} \leq \mathcal{D}^l_{10}$ and $\mathcal{D}^{k'l'}_{00} \leq \mathcal{D}^{kl}_{00}$. We now prove a few more similar properties.

\begin{lemma}\label{lem1}
Let $G \in \mathcal{MOP}$. For the given MOP-partition $(G_L,G_R,uv)$ the following holds:

1. For all $k,l \geq 0$ we have $\mathcal{D}_{11} \geq \max(\mathcal{D}^k_{10},\mathcal{D}^k_{01})$ and $ \min(\mathcal{D}^k_{10},\mathcal{D}^k_{01}) \geq \mathcal{D}^{kl}_{00}.$

2. If $deg_{G_R}(u) \geq 2$, then $2 \cdot \mathcal{D}^3_{10} \geq \mathcal{D}^4_{10}$ and $2 \cdot \mathcal{D}^3_{01} \geq \mathcal{D}^4_{01}$.

3. If $deg_{G_R}(u) = 3$, then $3 \cdot \mathcal{D}^2_{01} \geq \mathcal{D}^3_{01}$. Moreover, if $deg_{G_R}(u) \geq 4$, then $2 \cdot \mathcal{D}^2_{01} \geq \mathcal{D}^3_{01}$. 
\end{lemma}

\begin{proof}
\textbf{Statement 1.} We show that for all $k \geq 0$ the inequality $\mathcal{D}_{11} \geq \mathcal{D}^k_{10}$ holds. Consider the function $$F: \mathfrak{D}_4(G_R \setminus v) \longrightarrow \mathfrak{D}_4(G_R),$$ which maps a 4-DS $D$ of $G_R \setminus v$ into a 4-DS $D \cup \{v\}$ of $G_R$. Clearly, $F$ is injective, therefore $\partial_4(G_R \setminus v) \leq \partial_4(G)$, this implies the inequality. The inequalities $\mathcal{D}_{11} \geq \mathcal{D}^k_{01}$ and $\min(\mathcal{D}^k_{10},\mathcal{D}^k_{01}) \geq \mathcal{D}^{kl}_{00}$ are easy to prove using the same approach.

\textbf{Statement 2.} The inequality $deg_{G_R}(u) \geq 2$ means that the subgraph $G_R$ has at least 3 vertices, thus $deg_{G_R}(v) \geq 2$. We denote by $w$ the common neighbor of $u$ and $v$, such that $w \in V(G_R)$ (since $uv$ is an outer edge of $G_R \in \mathcal{MOP}$, there is exactly one such vertex).  Our goal is to show that $ \mathcal{D}^4_{10} - \mathcal{D}^3_{10} \leq \mathcal{D}^3_{10}$. The left hand side equals to the number of 4-DS of the graph $G_R \setminus v$ such that they don't have vertices from $N_{G_R}[u] \setminus v$. The right hand side equals the number of 4-DS of the same graph which contain at least 1 vertex from the set $N_{G_R}[u] \setminus v$, therefore the inequality holds. It is easy to prove that $2\mathcal{D}^3_{10} \geq \mathcal{D}^4_{10}$, using the same approach.

\textbf{Statement 3.} By the definition, $\mathcal{D}^2_{01}$ is the number of 4-DS of $G_R \setminus u$ with at least two vertices from the set $N_{G_R}[u] \setminus v$. Therefore, the difference $\mathcal{D}^3_{01} - \mathcal{D}^2_{01}$ equals to the number of 4-DS of the graph $G_R \setminus u$, with exactly one vertex from the set $N_{G_R}[u] \setminus v$. Let $N_{G_{R}}[u] \setminus v = \{w_1,\ldots w_s\}$, where $s \geq 2$. If $s = 2$, then the number of 4-DS with both vertices $w_1$ and $w_2$ is at least half of the number of 4-DS with one fixed vertex, this yields the inequality $\mathcal{D}^3_{01} - \mathcal{D}^2_{01} \geq 2 \cdot \mathcal{D}^2_{01}$. If $s \geq 3$, then it is easy to see that the number of 4-DS with exactly one vertex from the set $N_{G_{R}}[u] \setminus v$ is less than then the number of 4-DS with at least two vertices, therefore  $\mathcal{D}^3_{01} - \mathcal{D}^2_{01} \geq \mathcal{D}^2_{01}$, as required.
\end{proof}

\section{Proof of Theorem 1}

We call a graph $G \in \mathcal{MOP}$ \textit{critical}, if $i(G) \leq \partial_4(G)$ and for every outerplanar graph $G'$ such that $|V(G')| < |V(G)|$ we have $i(G') > \partial_4(G')$. In this section we show that there are no critical graphs, therefore Theorem~\ref{thm1} holds. It suffices to consider only maximal outerplanar graphs, since for every graph $G_0 \in \mathcal{OP}$ there exists a graph $G \in \mathcal{MOP}$ such that $G_0$ is a spanning subgraph of $G$ and the inequalities $i(G_0) > i(G)$ and $\partial_4(G_0) \leq \partial_4(G)$ hold. 

Therefore, we consider a graph $G \in \mathcal{MOP}$ and its weak dual graph $T(G)$ (remind that $T(G)$ is a subcubic tree). Let $x_1x_2\ldots x_k$ be some diametral path in $T(G)$. If $k \leq 3$, then there are only 3 possible MOPs up to isomorphism and it is easy to check that they are not critical. Thus we assume that $k \geq 4$.

\begin{lemma}\label{lem2}
If a graph $G \in \mathcal{MOP}$ has an edge $uv$ such that $deg(u) = 2$ and $deg(v) = 3$, then $G$ is not critical. 
\end{lemma}

\begin{proof}
Since $G \in \mathcal{MOP}$, the vertices $u$ and $v$ have the unique common neighbor $a$ and the vertices $v$ and $a$ have the unique common neighbor~$b$, other then $u$. Let $G_1 = G \setminus u$ and $G_3 = G \setminus \{u,a,v\}$. Then $$i(G) = i(G,u^-) + i(G,u^+) = i(G_1) + i(G_3).$$ 
We now show that
$$\partial_4(G) - \partial_4(G_1) \leq \partial_4(G_3).$$ 

The difference $\partial_4(G) - \partial_4(G_1)$ equals to the number of 4-DS $D$ of the graph $G$ such that $D \setminus u$ is not a 4-DS for the graph $G_1$. Since $v$ belongs to every 4-DS in both $G$ and $G_1$, this is possible if and only if $a \notin D$ and exactly two vertices from the set $N(a) \setminus \{u,v\}$ belong to $D$. Let $\mathfrak{D}'_4(G_1)$ be the family of 4-DS $D'$ of $G$ such that $D' \setminus u$ is a 4-DS of $G_1$. Consider the function $$F: (\mathfrak{D}_4(G) \setminus \mathfrak{D}'_4(G_1)) \longrightarrow \mathfrak{D}_4(G_3),$$
such that $F(D) = (D \cup \{b\}) \setminus \{u,v\}$. It is easy to see that $F$ is injective, because if $D',D'' \in \mathfrak{D}_4(G) \setminus \mathfrak{D}'_4(G_1)$ are two distinct 4-DS of $\mathfrak{D}_4(G)$, then the sets $D' \setminus \{u,v,b\}$ and $D'' \setminus \{u,v,b\}$ are also distinct. Therefore, $$\partial_4(G) \leq \partial_4(G_1) + \partial_4(G_3) < i(G_1) + i(G_3) = i(G)$$
and $G$ is not critical.
\end{proof}

Lemma~\ref{lem2} implies that every support vertex of $T(G)$ has degree 3. In partucular, $deg(x_2) = 3$ and $f_2$ is adjacent to some end faces $f_1$ and $f'_1$. In the rest of the chapter we denote the faces $f_1$, $f'_1$, $f_2$ and $f_3$ by $a_1a_2b_1$, $a_2a_3b_2$, $b_1b_2a_2$ and $b_1b_2c_1$ respectively. 

\begin{lemma}\label{lem3}
If $deg(x_2) = deg(x_3) = deg(x'_2) = 3$, then $G$ is not critical.
\end{lemma}

\begin{proof}
Suppose that $b_2c_1$ is the common edge of the faces $f_3$ and $f_4$. Let $G_L = G[f_3;f_4]$. Consider the $(G_L,G_R,b_2c_1)$-partition of $G$. Since $G$ is critical, we have $i(G_R) > \partial_4(G_R)$. Our goal is to find a constant $c > 0$ such that $i(G) \geq c \cdot i(G_R)$ and $\partial_4(G) \leq c \cdot \partial_4(G_R)$. It is easy to check that the following holds:
$$i(G) = i(G_L) \cdot \mathcal{I}_{00} + i(G_L,c^+_1) \cdot \mathcal{I}_{01} + i(G_L,b^+_2) \cdot \mathcal{I}_{10} = 29 \cdot \mathcal{I}_{00} + 10 \cdot \mathcal{I}_{01} + 10 \cdot \mathcal{I}_{10}.$$
Since $\mathcal{I}_{00} \geq \max(\mathcal{I}_{01},\mathcal{I}_{10})$, we have

$$i(G) > 16 \cdot (\mathcal{I}_{00} + \mathcal{I}_{01} + \mathcal{I}_{10}) = 16 \cdot i(G_R).$$
Moreover,
$$\partial_4(G) = \partial_4(G,b_2^+,c_1^+) + \partial_4(G,b_2^+,c_1^-) + \partial_4(G,b_2^-,c_1^+) + \partial_4(G,b_2^-,c_1^-) $$
$$= 5 \cdot \mathcal{D}_{11} + (2 \cdot  \mathcal{D}^{4}_{10} + \mathcal{D}^{3}_{10}) + (2 \cdot  \mathcal{D}^{4}_{01} + \mathcal{D}^{3}_{01}) + (\mathcal{D}^{44}_{00} + \mathcal{D}^{33}_{00}).$$

By Lemma~\ref{lem1}, we have $\partial_4(G) \leq  13 \cdot \mathcal{D}_{11} < 16 \cdot \partial_4(G_R)$. Therefore, $G$ is not critical.
\end{proof}

\begin{lemma}\label{lem4}
If $deg(x_3) = 3$ and $f_3$ is adjacent to some end face $f'_2$, then $G$ is not critical.
\end{lemma}
\begin{proof}
Denote the face $f_4$ by $c_1b_2c_2$ and let $G_L = G[f_3;f_4]$. Consider the $(G_L,G_R,b_2c_1)$-partition of graph $G$ and the $(G'_L,G_R,b_2c_1)$-partition of graph $G_2 = \{a_1,a_3\}$, where $G'_L = G_L \setminus \{a_1,a_3\}$. Again, our goal is to find a constant $c > 0$ such that $i(G) \geq c \cdot i(G_2)$ and $\partial_4(G) \leq c \cdot \partial_4(G_2)$. We have
$$i(G) = i(G_L) \cdot \mathcal{I}_{00} + i(G_L,c^+_1) \cdot \mathcal{I}_{01} + i(G_L,b^+_2) \cdot \mathcal{I}_{10}  = 12 \cdot \mathcal{I}_{00} + 5 \cdot \mathcal{I}_{01} + 4 \cdot \mathcal{I}_{10};$$
$$i(G_2) = i(G'_L) \cdot \mathcal{I}_{00} + i(G'_L,c^+_1) \cdot \mathcal{I}_{01} + i(G'_L,b^+_2) \cdot \mathcal{I}_{10}  = 5 \cdot \mathcal{I}_{00} + 2 \cdot \mathcal{I}_{01} + 2 \cdot \mathcal{I}_{10}.$$
Since $\mathcal{I}_{00} \geq \max(\mathcal{I}_{01},\mathcal{I}_{10})$, we have $i(G) > \frac{16}{7} \cdot i(G_2)$. Moreover,
$$\partial_4(G) = 3 \cdot \mathcal{D}_{11} + (2 \cdot \mathcal{D}^3_{10} + \mathcal{D}^2_{10}) + (\mathcal{D}^4_{01} + \mathcal{D}^3_{01})  + \mathcal{D}^{32}_{00};$$
$$\partial_4(G_2) = 2 \cdot \mathcal{D}_{11} + \mathcal{D}^3_{10} + \mathcal{D}^3_{01} + \mathcal{D}^{22}_{00}.$$
We now show that
$$3 \cdot \mathcal{D}_{11} + (2 \cdot \mathcal{D}^3_{10} + \mathcal{D}^2_{10}) + (\mathcal{D}^4_{01} + \mathcal{D}^3_{01})  + \mathcal{D}^{32}_{00} \leq \frac{16}{7} \cdot (2 \cdot \mathcal{D}_{11} + \mathcal{D}^3_{10} + \mathcal{D}^3_{01} + \mathcal{D}^{22}_{00}).$$
It is sufficient to prove the inequality
$$\mathcal{D}^2_{10} + \mathcal{D}^4_{01} + \mathcal{D}^{32}_{00} \leq \frac{11}{7} \cdot \mathcal{D}_{11} + \frac{2}{7} \cdot \mathcal{D}^3_{10} + \frac{9}{7}\cdot \mathcal{D}^3_{01}.$$

If $|V(G_R)| = 2$, then $\mathcal{D}^2_{10} = \mathcal{D}^{32}_{00} = 0$ and we are done. Otherwise by Lemma~1 we have $2 \cdot \mathcal{D}^3_{01} \geq \mathcal{D}^4_{01}$ and $\mathcal{D}^3_{01} \geq \mathcal{D}^{32}_{00}$, therefore $$\mathcal{D}^2_{10} + \frac 67 \cdot \mathcal{D}^4_{01} \leq \frac{11}{7} \cdot \mathcal{D}_{11} + \frac 27 \cdot  \mathcal{D}^3_{10} .$$ This completes the proof.
\end{proof}

In the rest of the chapter we assume that $deg(x_3) = 2$ and denote the face $f_4$ by $b_2c_1c_2$.

\begin{lemma}\label{lem5}
If $deg(x_4) = 3$, then $G$ is not critical.
\end{lemma}

\begin{proof}
Let $G_L = G[f_4;f_5]$. and $f'_3$ be the face adjacent to $f_4$, other than $f_3$ and $f_5$. We have two cases depending on the location of $f'_3$.

\textbf{Case 1.}  $f'_3$ contains the edge $b_2c_2$.  We consider the $(G_L,G_R,c_1c_2)$-partition of $G$. 

\textbf{Subcase 1.} $deg(x'_3) = 1$. Let $G_3 = G \setminus \{a_1,a_2,a_3\}$. Consider the $(G'_L,G_R,c_1c_2)$-partition of $G_3$, where $G'_L = G_L \setminus \{a_1,a_2,a_3\}$. We have
$$i(G) = 16 \cdot \mathcal{I}_{00} + 7\cdot \mathcal{I}_{01} + 10\cdot \mathcal{I}_{10} > 3 \cdot (5\cdot \mathcal{I}_{00} + 2\cdot \mathcal{I}_{01} + 2\cdot \mathcal{I}_{10}) = 3 \cdot i(G_3);$$
$$\partial_4(G) = 4 \cdot \mathcal{D}_{11} + (2\cdot \mathcal{D}^3_{10} +  \mathcal{D}^2_{10}) + (3\cdot \mathcal{D}^3_{01} +  \mathcal{D}^2_{01}) + (2\cdot \mathcal{D}^{22}_{00}+ \mathcal{D}^{11}_{00});$$
$$\partial_4(G_3) = 2\cdot \mathcal{D}_{11} + \mathcal{D}^3_{10} + \mathcal{D}^3_{01} + \mathcal{D}^{22}_{00}.$$

Clearly, $\partial_4(G) \leq 3 \cdot \partial_4(G_3) < 3 \cdot i(G_3) \leq i(G)$, as required.

\textbf{Subcase 2.} $deg(x'_3) = 3$. By the previous lemma, $f'_3$ is adjacent to some end faces $f'_2$ and $f''_2$.  We have
$$i(G) = 39\cdot \mathcal{I}_{00} + 14 \cdot \mathcal{I}_{01} + 25 \cdot \mathcal{I}_{10} \geq 26 \cdot ( \mathcal{I}_{00} +  \mathcal{I}_{01} +  \mathcal{I}_{10}) = 26 \cdot i(G_R);$$
$$\partial_4(G) = 7 \cdot \mathcal{D}_{11} + (3 \cdot \mathcal{D}^4_{10} + \mathcal{D}^3_{10}) + (4 \cdot \mathcal{D}^4_{01} + \mathcal{D}^3_{01}) + (2 \cdot \mathcal{D}^{12}_{00} + \mathcal{D}^{01}_{00}).$$
By Lemma~1, we have $\partial_4(G)  \leq 19 \cdot \mathcal{D}_{11} < 26 \cdot  \partial_4(G_R)$.

\textbf{Subcase 3.} $deg(x'_3) = 2$ and $deg(b_2) = 6$. In this case $f'_3$ is adjacent to a face $f'_2$ which is adjacent to end faces $f'_1$ and $f''_1$. Therefore,
$$i(G) = 59\cdot \mathcal{I}_{00} + 14\cdot \mathcal{I}_{01} + 35\cdot \mathcal{I}_{10} \geq 36 \cdot (\mathcal{I}_{00} + \mathcal{I}_{01} + \mathcal{I}_{10}) = 36 \cdot i(G_R);$$
$$\partial_4(G) = 11 \cdot \mathcal{D}_{11} + (6 \cdot \mathcal{D}^3_{01} + 2 \cdot \mathcal{D}^2_{01}) + (3 \cdot \mathcal{D}^4_{10} + \mathcal{D}^3_{10}) + 2 \cdot (\mathcal{D}^{24}_{00} + \mathcal{D}^{13}_{00})$$ $$\leq 26 \cdot \mathcal{D}_{11} = 26 \cdot \partial_4(G_R).$$

\textbf{Subcase 4.} $deg(x'_3) = 2$ and $deg(b_2) = 8$. Again, $f'_3$ is adjacent to a face $f'_2$ which is adjacent to end faces $f'_1$ and $f''_1$. Therefore,
$$i(G) = 53\cdot \mathcal{I}_{00} + 35\cdot \mathcal{I}_{01} + 35\cdot \mathcal{I}_{10} \geq 41 \cdot (\mathcal{I}_{00} + \mathcal{I}_{01} + \mathcal{I}_{10}) = 41 \cdot i(G_R);$$
$$\partial_4(G) \leq 4 \cdot \partial_4(G,b^+_2,c^+_2) = 4 \cdot 10 \cdot \mathcal{D}_{11} \leq 40 \cdot i(G_R).$$

\textbf{Case 2.} The face $f'_3$ contains the edge $c_1c_2$.  Consider the $(G_L,G_R,b_2c_2)$-partition of $G$. 

\textbf{Subcase 1.} $deg(x'_3) = 1$.  Let $G_3 = G \setminus \{a_1,a_2,a_3\}$. Consider the $(G'_L,G_R,b_2c_2)$-partition of $G_3$, where $G'_L = G_L \setminus \{a_1,a_2,a_3\}$. We have
$$i(G) = 19 \cdot \mathcal{I}_{00} + 7 \cdot \mathcal{I}_{01} + 4 \cdot \mathcal{I}_{10} > 3 \cdot (5 \cdot \mathcal{I}_{00} + 2 \cdot \mathcal{I}_{01} + 2 \cdot \mathcal{I}_{10}) = 3 \cdot i(G_3);$$
$$\partial_4(G) = 5 \cdot \mathcal{D}_{11} + 3 \cdot \mathcal{D}^3_{10} + \mathcal{D}^4_{01} +  \mathcal{D}^{42}_{00} \leq 3 \cdot (2 \cdot \mathcal{D}_{11} +  \mathcal{D}^3_{10} +  \mathcal{D}^3_{01} + \mathcal{D}^{22}_{00}) = 3 \cdot \partial_4(G_3).$$

By Lemma~\ref{lem1}, we have $\mathcal{D}_{11} + 3 \cdot \mathcal{D}^3_{01} > \mathcal{D}^{42}_{00} + \mathcal{D}^4_{01}$, as required.

\textbf{Subcase 2.} $deg(x'_3) = 3$. By the previous lemma, $f'_3$ is adjacent to end faces $f'_2$ and $f''_2$. We have
$$i(G) = 45 \cdot \mathcal{I}_{00} + 14 \cdot \mathcal{I}_{01} + 10 \cdot \mathcal{I}_{10} \geq 23 \cdot (\mathcal{I}_{00} + \mathcal{I}_{01} + \mathcal{I}_{10}) = 23 \cdot i(G_R);$$
$$\partial_4(G) = 8 \cdot \mathcal{D}_{11} + (3 \cdot \mathcal{D}^4_{10}  + 2 \cdot \mathcal{D}^3_{10}) + 3 \cdot \mathcal{D}^4_{01} + \mathcal{D}^{43}_{00} \leq 17 \cdot \mathcal{D}_{11} < 23 \cdot \partial_4(G_R).$$

\textbf{Subcase 3.} $deg(x'_3) = 2$ and $deg(c_1) = 4$. $f'_3$ is adjacent to a face $f'_2$ which is adjacent to end faces $f'_1$ and $f''_1$. We have
$$i(G) = 74 \cdot \mathcal{I}_{00} + 14 \cdot \mathcal{I}_{01} + 14 \cdot \mathcal{I}_{10} \geq 34 \cdot (\mathcal{I}_{00} + \mathcal{I}_{01} + \mathcal{I}_{10}) = 34 \cdot i(G_R);$$
$$\partial_4(G) = 13 \cdot \mathcal{D}_{11} + 3 \cdot \mathcal{D}^4_{10} + 3 \cdot \mathcal{D}^4_{01} + \mathcal{D}^{44}_{00}  \leq 20 \cdot \mathcal{D}_{11} < 34 \cdot \partial_4(G_R).$$

\textbf{Subcase 4.} $deg(x'_3) = 2$ and $deg(c_1) = 6$. $f'_3$ is adjacent to a face $f'_2$ which is adjacent to end faces $f'_1$ and $f''_1$. We have $$i(G) = 59 \cdot \mathcal{I}_{00} + 35 \cdot \mathcal{I}_{01} + 14 \cdot \mathcal{I}_{10} > 36 \cdot (\mathcal{I}_{00} + \mathcal{I}_{01} + \mathcal{I}_{10}) = 36 \cdot i(G_R);$$
$$\partial_4(G) \leq 3 \cdot \partial_4(G,b^+_2,c^+_2) + \partial_4(G,b^-_2,c^-_2) = 3 \cdot 11 \cdot \mathcal{D}_{11} + (2 \cdot \mathcal{D}^{42}_{00} + \mathcal{D}^{31}_{00}) $$ 

By Lemma~1, we have 
$\partial_4(G) \leq 36 \cdot \mathcal{D}_{11} \leq 36 \cdot \partial_4(G_R)$. 
\end{proof}

\begin{lemma}\label{lem6}
If $deg(x_3) = deg(x_4) = 2$ and $deg(c_1) = 3$, then $G$ is not critical.
\end{lemma}

\begin{proof}
If $deg(x_3) = deg(x_4) = 2$ and $deg(c_1) = 3$, then $f_5$ contains $b_2c_2$. Let $G_L = G[f_3;f_4]$. For the $(G_L,G_R,b_2c_1)$-partition of $G$ we have
$$i(G) = 7 \cdot \mathcal{I}_{00} + 5 \cdot \mathcal{I}_{01} + 2 \cdot \mathcal{I}_{10} > 4 \cdot (\mathcal{I}_{00} + \mathcal{I}_{01} + \mathcal{I}_{10}) = 4 \cdot i(G_R);$$
$$\partial_4(G) = \partial_4(G,b^+_2,c^+_1) + \partial_4(G,b^-_2,c^+_1) = 3 \cdot \mathcal{D}_{11} + \mathcal{D}^4_{01} \leq 4 \cdot  \mathcal{D}_{11} \leq \partial_4(G_R).$$

Therefore, $G$ is not critical.
\end{proof}

\begin{lemma}\label{lem7}
If $deg(x_3) = deg(x_4) = 2$ and $deg(x_5) = 3$, then $G$ is not critical.
\end{lemma}

\begin{proof}
We denote the faces $f_2, f_3, f_4$ by $a_2b_1b_2, b_1b_2c_1, b_2c_1c_2$ respectively. By the previous lemma, $deg(c_1) \geq 4$ and the face $f_5$ corresponds to the triangle $c_1c_2d_1$. Let $f'_4$ be the face adjacent to $f_5$, other than $f_4$ and $f_6$. Let $G_L = G[f_5;f_6]$. There are two possible cases depending on the location of $f'_4$ in $G$.

\textbf{Case 1.} The face $f'_4$ contains the edge $c_1d_1$. Let $d_0$ be the third vertex of $f'_4$. Consider the $(G_L,G_R,c_2d_1)$-partition of $G$.

\textbf{Subcase 1a.} $deg(x'_4) = 1$ and $\min(deg_G(c_2),deg_G(d_1)) \geq 5$.  Consider the graph $G_6 = G \setminus \{a_1,a_2,a_3,b_1,b_2,d_0\}$ and the $(G_L \cap V(G_6),G_R,c_2d_1)$-partition of $G_6$. We have
$$i(G) = 23 \cdot \mathcal{I}_{00} + 9 \cdot \mathcal{I}_{01} + 14 \cdot \mathcal{I}_{10} > 10 \cdot (2 \cdot \mathcal{I}_{00} +  \mathcal{I}_{01} +  \mathcal{I}_{10}) = 10 \cdot i(G_6).$$
We now prove that
$$\partial_4(G) = 7 \cdot \mathcal{D}_{11} + (4 \cdot \mathcal{D}^3_{10} + 2 \cdot \mathcal{D}^2_{10}) + (3 \cdot \mathcal{D}^3_{01} + 3 \cdot \mathcal{D}^2_{01}) + (3 \cdot \mathcal{D}^{22}_{00}+\mathcal{D}^{12}_{00}) \leq $$ $$10 \cdot ( \mathcal{D}_{11} +  \mathcal{D}^2_{10} +  \mathcal{D}^2_{01} +  \mathcal{D}^{11}_{00}) =  10 \cdot \partial_4(G_6).$$

It remains to show that $$3 \cdot \mathcal{D}_{11} + 8 \cdot \mathcal{D}^2_{10} + 7 \cdot \mathcal{D}^2_{01} \geq 4 \cdot \mathcal{D}^3_{10} + 3 \cdot \mathcal{D}^3_{01} + 3 \cdot \mathcal{D}^{22}_{00} +  \mathcal{D}^{12}_{00}.$$

Since $\min(deg(a),deg(c)) \geq 5$, we have $3 \cdot \mathcal{D}^2_{10} \geq \mathcal{D}^3_{10}$ and $3 \cdot \mathcal{D}^2_{01} \geq \mathcal{D}^3_{01}$ by Lemma~\ref{lem1}. Moreover, $\min(\mathcal{D}^2_{01}, \mathcal{D}^2_{10}) \geq \mathcal{D}^{22}_{00}$. If $deg(a) = 5$, then  $ \mathcal{D}^{12}_{00} = 0$ and we are done. If $deg(a) \geq 6$, then $2 \cdot \mathcal{D}^2_{01} \geq  \mathcal{D}^3_{01}$ and we are done. 

\textbf{Subcase 1b.} $deg(x'_4) = 1$ and $\min(deg(a),deg(c)) = 4$. Consider the graph $G_4 = G \setminus \{a_1,a_2,a_3,b_1\}$ and the $(G_L \cap V(G_4),G_R,c_2d_1)$-partition of $G_4$. We have
$$i(G) = 23 \cdot \mathcal{I}_{00} + 9 \cdot \mathcal{I}_{01} + 14 \cdot \mathcal{I}_{10} \geq \frac 92 \cdot (5 \cdot \mathcal{I}_{00} + 2 \cdot \mathcal{I}_{01} + 2 \cdot \mathcal{I}_{10}) = \frac 92 \cdot i(G_4).$$
We now prove that
$$\partial_4(G) = 7 \cdot \mathcal{D}_{11} + (4 \cdot \mathcal{D}^3_{10} + 2 \cdot \mathcal{D}^2_{10}) + (3 \cdot \mathcal{D}^3_{01} + 3 \cdot \mathcal{D}^2_{01}) + (3 \cdot \mathcal{D}^{22}_{00}+ \mathcal{D}^{11}_{00}) \leq$$ $$\frac 92 \cdot (2 \cdot \mathcal{D}_{11} +  \mathcal{D}^3_{10} +  \mathcal{D}^3_{01} + \mathcal{D}^{22}_{00}) = \frac 92 \cdot \partial_4(G_4).$$
It suffices to show that $$2 \cdot \mathcal{D}_{11} + \frac 12  \cdot \mathcal{D}^3_{10} + \frac 32 \cdot \mathcal{D}^3_{01} \geq 2 \cdot \mathcal{D}^2_{10} + 3 \cdot \mathcal{D}^2_{01}.$$ Since $\min(deg(a),deg(c)) = 4$, we have $\textrm{min}( \mathcal{D}^2_{10}, \mathcal{D}^2_{01}) = 0$, therefore the inequality holds.

\textbf{Subcase 2.} $deg(x'_4) = 3$. The face $f'_4$ is adjacent to some faces $f'_3$ and $f''_3$. By the previous lemmas, both $f'_3$ and $f''_3$ are end faces. Therefore,
$$i(G) = 55 \cdot \mathcal{I}_{00} + 18 \cdot \mathcal{I}_{01} + 35 \cdot \mathcal{I}_{10} \geq 36 \cdot ( \mathcal{I}_{00} +  \mathcal{I}_{01} +  \mathcal{I}_{10}) = 36 \cdot i(G_R);$$
$$\partial_4(G) \leq 2 \cdot \partial_4(G,c^+_2,d^+_2) + 2 \cdot \partial_4(G,c^+_2,d^-_2) =  2 \cdot (11 \cdot \mathcal{D}_{11} + 4 \cdot \mathcal{D}^4_{10} + 2 \cdot \mathcal{D}^3_{10})$$
$$\leq  34 \cdot \mathcal{D}_{11} < 36 \cdot \partial_4(G_R).$$

In the remaining subcases we consider the induced subgraph $G'$ of $G$ with the vertex set $(V(G) \setminus V(G_L)) \cup \{b_2,c_2,d_2\}$ and the $(G'_L,G_R,c_1d_1)$-partition of $G$, where $G'_L$ is an induced subgraph of $G_L$ with the vertex set $\{b_2,c_1,c_2,d_0,d_1\}$.

\textbf{Subcase 3.} $deg(x'_4) = deg(x'_3) = 2$. The face $f'_3$ is adjacent to some face $f'_2$. By Lemma~\ref{lem3} $f'_2$ is adjacent to end faces $f'_1$ and $f''_1$. Moreover, by the previous lemmas, the faces $f'_3$ and $f'_4$ does not contain any vertices of degree 3.

\textbf{Subcase 3a.} $deg(c_1) \geq 6$. If $deg(c_1) \geq 7$, then $deg(d_0) = 3$, a contradiction. Suppose that $deg(c_1) = 6$. 

$$i(G) = 106 \cdot \mathcal{I}_{00} + 63 \cdot \mathcal{I}_{01} + 63 \cdot \mathcal{I}_{10} > 20 \cdot (5 \cdot \mathcal{I}_{00} + 2 \cdot \mathcal{I}_{01} + 2 \cdot \mathcal{I}_{10}) = 20 \cdot i(G');$$
$$\partial_4(G) \leq 25 \cdot \mathcal{D}_{11} + (12 \cdot \mathcal{D}^3_{10} + 10 \cdot \mathcal{D}^2_{10}) + (12 \cdot \mathcal{D}^3_{01} + 10 \cdot \mathcal{D}^2_{01}) $$
$$+ 9 \cdot \mathcal{D}^{22}_{00} + 3 \cdot \mathcal{D}^{21}_{00} + 3 \cdot \mathcal{D}^{12}_{00} + 4 \cdot \mathcal{D}^{11}_{00}$$
$$\leq 25 \cdot \mathcal{D}_{11} + 22 \cdot \mathcal{D}^3_{10} + 22 \cdot \mathcal{D}^3_{01} + 19 \cdot \mathcal{D}^{22}_{00}$$
 $$< 20 \cdot (2 \cdot \mathcal{D}_{11} + \mathcal{D}^3_{10} + \mathcal{D}^3_{01} + \mathcal{D}^{22}_{00}) = 20 \cdot \partial_4(G').$$

\textbf{Subcase 3b.} $deg(c_1) = 5$. If $deg_{G_L}(d_1) \geq 5$, then $deg(d_0) = 3$ and we use Lemma~\ref{lem6}. 
Suppose that $deg_{G_L}(d_1) = 4$ and $deg_{G_L}(d_0) = 5$. We have
$$i(G) = 116 \cdot \mathcal{I}_{00} + 45 \cdot \mathcal{I}_{01} + 63 \cdot \mathcal{I}_{10} > 23 \cdot ( 5 \cdot \mathcal{I}_{00} +  2 \cdot \mathcal{I}_{01} + 2 \cdot \mathcal{I}_{10}) = 23 \cdot i(G');$$
$$\partial_4(G) = 28 \cdot \mathcal{D}_{11} + (8 \cdot \mathcal{D}^4_{10} + 14 \cdot \mathcal{D}^3_{10} + 2 \cdot \mathcal{D}^2_{10}) + (12 \cdot \mathcal{D}^3_{01} + 10 \cdot \mathcal{D}^2_{01})$$ 
$$+ (6 \cdot D^{23}_{00} + 3 \cdot D^{22}_{00} + 2 \cdot D^{13}_{00} + 1 \cdot D^{12}_{00})\leq $$
$$\leq 28 \cdot \mathcal{D}_{11} + 8 \cdot \mathcal{D}_{11} + 16 \cdot \mathcal{D}^3_{10} + 22 \cdot \mathcal{D}^3_{01} + 8 \cdot \mathcal{D}_{11} + 4 \cdot \mathcal{D}^{22}_{00}$$
$$ \leq 23 \cdot (2 \cdot \mathcal{D}_{11} + \mathcal{D}^3_{10} + \mathcal{D}^3_{01} + \mathcal{D}^{22}_{00}) \leq 23 \cdot \partial_4(G').$$

\textbf{Subcase 4.} $deg(x'_4) = 2$, $deg(x'_3) = 3$. This is possible only if $f'_3$ is adjacent to end faces, otherwise we use Lemma~\ref{lem5}.

\textbf{Subcase 4a.} $deg(c_1) = 7$. The face $f'_3$ is adjacent to two end faces by Lemmas~\ref{lem3} and~\ref{lem6}. We have
$$i(G) = 73 \cdot \mathcal{I}_{00} + 45 \cdot \mathcal{I}_{01} + 49 \cdot \mathcal{I}_{10} > 14 \cdot ( 5 \cdot \mathcal{I}_{00} +  2 \cdot \mathcal{I}_{01} + 2 \cdot \mathcal{I}_{10}) = 14 \cdot i(G');$$
$$\partial_4(G) \leq 15 \cdot \mathcal{D}_{11} + 15 \cdot \mathcal{D}^3_{10} + 15 \cdot \mathcal{D}^3_{01} + 15 \cdot \mathcal{D}^{22}_{00}$$ $$ < 14 \cdot (2 \cdot \mathcal{D}_{11} + \mathcal{D}^3_{10} + \mathcal{D}^3_{01} + \mathcal{D}^{22}_{00}) \leq 14 \cdot \partial_4(G').$$

\textbf{Subcase 4b.} $deg(c_1) = 5$, $deg(x'_4) = 2$, $deg(x'_3) = 2$. The face $f'_3$ is adjacent to two end faces by Lemmas~\ref{lem3} and~\ref{lem6}. We have
$$i(G) = 88 \cdot \mathcal{I}_{00} + 18 \cdot \mathcal{I}_{01} + 49 \cdot \mathcal{I}_{10} > 15 \cdot ( 5 \cdot \mathcal{I}_{00} +  2 \cdot \mathcal{I}_{01} + 2 \cdot \mathcal{I}_{10}) = 15 \cdot i(G');$$
$$\partial_4(G) \leq 18 \cdot \mathcal{D}_{11} + 6 \cdot \mathcal{D}^4_{10} + (9 \cdot \mathcal{D}^3_{01} + 9 \cdot \mathcal{D}^2_{01}) + (3 \cdot \mathcal{D}^{23}_{00} + \mathcal{D}^{13}_{00})$$ 
$$ < 15 \cdot (2 \cdot \mathcal{D}_{11} + \mathcal{D}^3_{10} + \mathcal{D}^3_{01} + \mathcal{D}^{22}_{00}) \leq 14 \cdot \partial_4(G').$$
 
\textbf{Case 2.} The face $f'_4$ corresponds to the triangle $c_2d_1d_2$.  We consider the $(G_L,G_R,c_1d_1)$-partition of $G$.

\textbf{Subcase 1.} $deg(x'_4) = 1$. Consider the graph $G_4 = G \setminus \{a_1,a_2,a_3,b_1\}$ and the $(G_L \cap V(G_4),G_R,c_2d_1)$-partition of $G_4$. We have
$$i(G) = 25 \cdot \mathcal{I}_{00} + 9 \cdot \mathcal{I}_{01} + 10 \cdot \mathcal{I}_{10} \geq \frac{34}{7} \cdot ( 5 \cdot \mathcal{I}_{00} + 2 \cdot \mathcal{I}_{01} + 2 \cdot \mathcal{I}_{10}) = \frac{34}{7} \cdot i(G_4).$$
Moreover,
$$\partial_4(G) = 7 \cdot \mathcal{D}_{11} + 4 \cdot \mathcal{D}^3_{10} +  (2 \cdot \mathcal{D}^4_{01} +   \mathcal{D}^3_{01}) + (2 \cdot \mathcal{D}^{32}_{00}+ \mathcal{D}^{22}_{00}) \leq $$ $$ \frac{34}{7} \cdot (2 \cdot \mathcal{D}_{11} +  \mathcal{D}^3_{10} +  \mathcal{D}^3_{01} +  \mathcal{D}^{22}_{00}) = \frac{34}{7} \cdot \partial_4(G_4).$$
It suffices to show that $$\frac{19}{7}\cdot \mathcal{D}_{11} + \frac{6}{7} \cdot \mathcal{D}^3_{10} + \frac{27}{7} \cdot \mathcal{D}^3_{01} \geq 2 \cdot \mathcal{D}^4_{01} + 2 \cdot \mathcal{D}^{32}_{00}.$$

By Lemma~1, we have $\mathcal{D}_{11} > \mathcal{D}^4_{01}$ and $\mathcal{D}^3_{01} > \mathcal{D}^{32}_{00}$, thus the inequality holds.

\textbf{Subcase 2.} $deg(x'_4) = 3$. We assume that $f'_4$ is adjacent to end faces $f'_3$ and $f''_3$ (it was shown in the previous case that the other configurations are not possible). Therefore,
$$i(G) = 59 \cdot \mathcal{I}_{00} + 18 \cdot \mathcal{I}_{01} + 25 \cdot \mathcal{I}_{10} \geq 34 \cdot ( \mathcal{I}_{00} +  \mathcal{I}_{01} +  \mathcal{I}_{10}) = 34 \cdot i(G_R);$$
$$\partial_4(G) = 12 \cdot \mathcal{D}_{11} + (4 \cdot \mathcal{D}^4_{10} + 3 \cdot \mathcal{D}^3_{10}) + (4 \cdot \mathcal{D}^4_{01} + 4 \cdot \mathcal{D}^4_{01}) + 2 \cdot \mathcal{D}^{33}_{00} + \mathcal{D}^{23}_{00} $$ $$ \leq 30 \cdot \mathcal{D}_{11} < 34 \cdot \partial_4(G_R).$$

In the remaining subcases we consider the induced subgraph $G''$ of $G$ with the vertex set $(V(G) \setminus V(G_L)) \cup \{b_2,c_2,d_2\}$ and the $(G'_L,G_R,c_1d_1)$-partition of $G$, where $G'_L$ is an induced subgraph of $G_L$ with the vertex set $\{b_2,c_2,d_2,c_1,d_1\}$.

\textbf{Subcase 3.} $deg(x'_4) = deg(x'_3) = 2$. As in the previous case, $f'_3$ is adjacent to some face $f'_2$ which is adjacent to end faces $f'_1$ and $f''_1$. Again, we assume that the faces $f'_3$ and $f'_4$ does not contain vertices of degree 3.

\textbf{Subcase 3a.} $deg(c_2) \geq 5$. If $deg(c_2) = 7$, then $deg(d_2) = 3$, a contradiction. Thus we assume that $deg(c_2) = 5$ and $d_2$ belongs to $f'_2$ and $f'_3$. Therefore,
$$i(G) = 116 \cdot \mathcal{I}_{00} + 63 \cdot \mathcal{I}_{01} + 45 \cdot \mathcal{I}_{10} \geq 23 \cdot (5 \cdot \mathcal{I}_{00} + 2 \cdot \mathcal{I}_{01} + 2 \cdot \mathcal{I}_{10}) = 20 \cdot i(G'');$$
$$\partial_4(G) \leq 24 \cdot \mathcal{D}_{11} + 12 \cdot (\mathcal{D}^4_{10} + \mathcal{D}^3_{10}) + 8 \cdot (\mathcal{D}^4_{01} + \mathcal{D}^3_{01})+ 12 \cdot \mathcal{D}^{32}_{00}$$
$$<  20 \cdot (2 \cdot \mathcal{D}_{11} + \mathcal{D}^3_{10} + \mathcal{D}^3_{01} + \mathcal{D}^{22}_{00}) = 20 \cdot \partial_4(G'').$$

\textbf{Subcase 3b.} $deg(c_2) = 4$. If $deg(d_2) = 3$, the we apply Lemma. We assume that $deg(d_2) = 5$. Therefore,
$$i(G) = 130 \cdot \mathcal{I}_{00} + 45 \cdot \mathcal{I}_{01} + 45 \cdot \mathcal{I}_{10} > 24 \cdot (5 \cdot \mathcal{I}_{00} + 2 \cdot \mathcal{I}_{01} + 2 \cdot \mathcal{I}_{10}) = 24 \cdot i(G'');$$
$$\partial_4(G) = 25 \cdot \mathcal{D}_{11} + (8 \cdot \mathcal{D}^4_{10} + 4 \cdot \mathcal{D}^3_{10}) + (8 \cdot \mathcal{D}^4_{01} + 4 \cdot \mathcal{D}^3_{01}) + (4 \cdot \mathcal{D}^{22}_{00} + 2 \cdot \mathcal{D}^{32}_{00} + 2 \cdot \mathcal{D}^{23}_{00} + \mathcal{D}^{22}_{00});$$
$$\partial_4(G) \leq 24 \cdot (2 \cdot \mathcal{D}_{11} + \mathcal{D}^3_{10} + \mathcal{D}^3_{01} + \mathcal{D}^{22}_{00}) = 24 \cdot \partial_4(G'').$$

\textbf{Subcase 4.} $deg(x'_4) = 2$, $deg(x'_3) = 3$. As in the previous case, it is possible only if $f'_3$ is adjacent to two end faces.

\textbf{Subcase 4a.} $deg(c_2) = 6$.  We have
$$i(G) = 77 \cdot \mathcal{I}_{00} + 45 \cdot \mathcal{I}_{01} + 35 \cdot \mathcal{I}_{10} > 15 \cdot ( 5 \cdot \mathcal{I}_{00} +  2 \cdot \mathcal{I}_{01} + 2 \cdot \mathcal{I}_{10}) = 15 \cdot i(G'');$$
$$\partial_4(G) \leq 16 \cdot \mathcal{D}_{11} + (8 \cdot \mathcal{D}^3_{10} + 3 \cdot \mathcal{D}^2_{10}) + (6 \cdot \mathcal{D}^4_{01} + 6 \cdot \mathcal{D}^3_{01} + \mathcal{D}^2_{01}) + (4 \cdot \mathcal{D}^{32}_{00} + 2 \cdot \mathcal{D}^{22}_{00} + 2 \cdot \mathcal{D}^{21}_{00} + \mathcal{D}^{11}_{00})$$ $$ \leq 16 \cdot \mathcal{D}_{11} + 11 \cdot \mathcal{D}^3_{10} + 13 \cdot \mathcal{D}^4_{01} + 9 \cdot \mathcal{D}^{32}_{00} \leq 15 \cdot (2 \cdot \mathcal{D}_{11} + \mathcal{D}^3_{10} + \mathcal{D}^3_{01} + \mathcal{D}^{22}_{00}) \leq 15 \cdot \partial_4(G'').$$

\textbf{Subcase 4b.} $deg(c_2) = 4$.  We have
$$i(G) = 98 \cdot \mathcal{I}_{00} + 18 \cdot \mathcal{I}_{01} + 35 \cdot \mathcal{I}_{10} > 16 \cdot ( 5 \cdot \mathcal{I}_{00} +  2 \cdot \mathcal{I}_{01} + 2 \cdot \mathcal{I}_{10}) = 16 \cdot i(G'');$$
$$\partial_4(G) \leq 18 \cdot \mathcal{D}_{11} + 4 \cdot \mathcal{D}^4_{10} + (6 \cdot \mathcal{D}^4_{01} + 3 \cdot \mathcal{D}^3_{01}) + 2 \cdot (\mathcal{D}^{34}_{00} + \mathcal{D}^{24}_{00}) $$
$$\leq 18 \cdot \mathcal{D}_{11} + 4 \cdot \mathcal{D}^4_{10} + 9 \cdot \mathcal{D}^4_{01} + 3 \cdot \mathcal{D}^{34}_{00}$$ $$ < 16 \cdot (2 \cdot \mathcal{D}_{11} + \mathcal{D}^3_{10} + \mathcal{D}^3_{01} + \mathcal{D}^{22}_{00}) \leq 16 \cdot \partial_4(G'').$$
\end{proof}

\begin{lemma}\label{lem8}
If $deg(x_3) = deg(x_4) = deg(x_5) = deg(x_6) = 2$, then $G$ is not critical.
\end{lemma}

\begin{proof}
We denote the faces $f_2, f_3, f_4$ by $a_2b_1b_2, b_1b_2c_1, b_2c_1c_2$ respectively. By Lemma~\ref{lem6}, we have $deg(c_1) \geq 4$. There are three possible cases.  

\textbf{Case 1.} $deg(c_1) \geq 5$. In this case $deg(c_2) = 3$, the face $f_5$ contains $c_1c_2$ and $c_1$ belongs to $f_6$. Let $G_L = G[f_4;f_5]$. Consider the $(G_L,G_R,c_1c_2)$-partition of $G$.
$$i(G) = 9 \cdot \mathcal{I}_{00} + 7 \cdot \mathcal{I}_{01} + 5 \cdot \mathcal{I}_{10} \geq 7 \cdot (\mathcal{I}_{00} + \mathcal{I}_{01} + \mathcal{I}_{10}) = 7 \cdot i(G_R);$$
$$\partial_4(G) = 4 \cdot \mathcal{D}_{11} + (2 \cdot \mathcal{D}^3_{01} + \mathcal{D}^2_{01}) \leq 7 \cdot \mathcal{D}_{11} \leq 7 \cdot \partial_4(G_R).$$

In the remaining cases  we denote the faces $f_5$ and $f_6$ by $c_1c_2d_1$ and $c_2d_1d_2$ respectively.

\textbf{Case 2.} $deg(c_2) \geq 5$. In this case $deg(d_1) = 3$. Let $G_L = G[f_5;f_6]$. Consider the $(G_L,G_R,c_2d_1)$-partition of $G$.
$$i(G) = 14 \cdot \mathcal{I}_{00} + 9 \cdot \mathcal{I}_{01} + 7 \cdot \mathcal{I}_{10} > 10 \cdot (\mathcal{I}_{00} + \mathcal{I}_{01} + \mathcal{I}_{10}) = 10 \cdot i(G_R);$$
$$\partial_4(G) = \partial_4(G,c^+_2,d^+_1) + \partial_4(G,c^-_2,d^+_1) = 6 \cdot \mathcal{D}_{11} + 3 \cdot \mathcal{D}^3_{01} + \mathcal{D}^2_{01}  \leq \partial_4(G_R).$$

\textbf{Case 3.} $deg(c_1) =  deg(c_2) = 4$. Let $G_L = G[f_4;f_5]$. Consider the $(G_L,G_R,c_1c_2)$-partition of $G$ and the $(G'_L, G_R,c_1c_2)$-partition of the graph $G_3 = G \setminus \{a_1,a_2,a_3\}$, where $G'_L = G_L \setminus \{a_1,a_2,a_3\}$. 

First, we show that $\mathcal{I}_{00} \leq \mathcal{I}_{01} + \mathcal{I}_{10}$. Denote by $G'_R$ the induced subgraph of $G_R$ with the vertex set $V(G_R) \setminus \{c_1,c_2\}$. Clearly,
$$\mathcal{I}_{00} = i(G'_R), \ \mathcal{I}_{10} = i(G'_R,d^-_1), \ \mathcal{I}_{01} = i(G'_R,d^-_1,d^-_2).$$
Therefore, $\mathcal{I}_{10} > \mathcal{I}_{01}$ and $\mathcal{I}_{00} - \mathcal{I}_{10} = i(G'_R,d^+_1) \leq \mathcal{I}_{01}$. We have
$$i(G) = 9 \cdot \mathcal{I}_{00} + 7 \cdot \mathcal{I}_{01} +  5 \cdot \mathcal{I}_{10} \geq \frac{10}{3} \cdot (3 \cdot \mathcal{I}_{00} + 2 \cdot \mathcal{I}_{01} + 1 \cdot \mathcal{I}_{10}) = \frac{10}{3} \cdot i(G_1).$$
Moreover,
$$\partial_4(G) = 4 \cdot \mathcal{D}_{11} + 3 \cdot \mathcal{D}^2_{10} + 2 \cdot D^3_{01}, \ \partial_4(G') =   \mathcal{D}_{11} + \mathcal{D}^2_{10} + D^3_{01}. $$

It remains to show that $\partial_4(G) \leq \frac{10}{3} \cdot \partial_4(G')$ or $\mathcal{D}_{11} \leq 2 \cdot \mathcal{D}^3_{01}$. Indeed, $\mathcal{D}_{11} - \mathcal{D}^3_{01}$ equals to the number of 4-DS of $G_R$ which contain $c_1$ and $c_2$ and does not contain $d_1$ and $\mathcal{D}^3_{01}$ equals to the number of 4-DS which contain $c_1,c_2$ and $d_1$. For every 4-DS $D$ of $G_R$ such that $d_1 \notin D$, the set $D \cup \{d_1\}$ is also a 4-DS of $G_R$, thus the inequality holds.
\end{proof}

\begin{lemma}\label{lem9}
If $deg(x_3) = deg(x_4) = deg(x_5) = 2$ and $deg(x_6) = 3$, then $G$ is not critical.
\end{lemma}
\begin{proof}
Denote the faces $f_5$ and $f_6$ by $c_1c_2d_1$ and $c_2d_1d_2$ respectively. Let $f'_5$ be the face adjacent to $f_6$ other than $f_5$ and $f_7$. By Lemma~\ref{lem8}, $deg(c_1) = 4$ and $deg(c_2) \geq 5$. Let $G_L = G[f_4;f_5]$.  We assume that $G_L$ contains at most one face of degree 3 except $f_2$, which is adjacent to two end faces. By previous lemmas, $G_L$ contains no vertices of degree 3. There are two possible cases depending on the location of $f'_5$ in $G$.

\textbf{Case 1.} $f'_5$ contains the edge $c_2d_2$. Denote by $c_3$ the third vertex of $f'_5$. In each of the following subcases we consider the partition $(G_L,G_R,d_1d_2)$ of the graph $G$ and the partition $(G'_L,G_R,d_1d_2)$ of the graph $G'$, where $G'$ is a spanning subgraph of $G$ with the vertex set $(V(G) \setminus V(G_L)) \cup \{c_1,c_2,c_3,d_1,d_2\}$. Clearly,
$$\partial_4(G') = 2 \cdot \mathcal{D}_{11} + \mathcal{D}^3_{10} + \mathcal{D}^3_{01} + \mathcal{D}^{22}_{00}.$$ It is not hard to check using Lemma 1, that in all subcases below we have $i(G) > \partial_4(G)$, therefore $G$ is not critical.

\textbf{Subcase 1.} $deg(x'_5) = 1$.
$$i(G) = 35 \cdot \mathcal{I}_{00} + 14 \cdot \mathcal{I}_{01} + 18 \cdot \mathcal{I}_{10} > 7 \cdot (5 \cdot \mathcal{I}_{00} +  2 \cdot \mathcal{I}_{01} +  2 \cdot \mathcal{I}_{10}) = 7 \cdot i(G');$$
$$\partial_4(G) = 10 \cdot \mathcal{D}_{11} + (6 \cdot \mathcal{D}^3_{10} + 3 \cdot \mathcal{D}^2_{10}) + (4 \cdot \mathcal{D}^3_{01} + 4 \cdot \mathcal{D}^2_{01}) + 4 \cdot \mathcal{D}^{22}_{00}.$$

\textbf{Subcase 2.} $deg(x'_5) = 3$ and $f'_5$ is adjacent to some faces $f'_4$ and $f''_4$. By the previous lemmas, both $f'_4$ and $f''_4$ are end faces. Therefore,
$$i(G) = 70 \cdot \mathcal{I}_{00} + 28 \cdot \mathcal{I}_{01} + 45 \cdot \mathcal{I}_{10} > 14 \cdot (5 \cdot \mathcal{I}_{00} +  2 \cdot \mathcal{I}_{01} +  2 \cdot \mathcal{I}_{10}) = 7 \cdot i(G');$$
$$\partial_4(G) = 22 \cdot \mathcal{D}_{11} + (6 \cdot \mathcal{D}^4_{10} + 4 \cdot \mathcal{D}^3_{10}) + (8 \cdot \mathcal{D}^3_{01} + 4 \cdot \mathcal{D}^2_{01}) + (4 \cdot \mathcal{D}^{23}_{00} + 4 \cdot \mathcal{D}^{12}_{00}).$$

\textbf{Subcase 3.} $deg(x'_5) = 2$ and $f'_4$ is adjacent to end faces $f'_3$ and $f''_3$. Two configurations are possible:

\textbf{Subcase 3a.} $deg(c_2) = 7$ and $c_2$ belong to $f'_4$.
$$i(G) = 98 \cdot \mathcal{I}_{00} + 70 \cdot \mathcal{I}_{01} + 63 \cdot \mathcal{I}_{10} > 19 \cdot (5 \cdot \mathcal{I}_{00} +  2 \cdot \mathcal{I}_{01} +  2 \cdot \mathcal{I}_{10}) = 19 \cdot i(G');$$
$$\partial_4(G) \leq 22 \cdot (\mathcal{D}_{11} + \mathcal{D}^3_{10} + \mathcal{D}^3_{01} + \mathcal{D}^{22}_{00}) \leq 19 \cdot (2 \cdot \mathcal{D}_{11} + \mathcal{D}^3_{10} + \mathcal{D}^3_{01} + \mathcal{D}^{22}_{00}).$$

\textbf{Subcase 3b.} $deg(c_2) = 5$ and $c_2$ does not belong to $f'_4$.
$$i(G) = 98 \cdot \mathcal{I}_{00} + 28 \cdot \mathcal{I}_{01} + 63 \cdot \mathcal{I}_{10} > 18 \cdot (5 \cdot \mathcal{I}_{00} +  2 \cdot \mathcal{I}_{01} +  2 \cdot \mathcal{I}_{10}) = 7 \cdot i(G');$$
$$\partial_4(G) = 26 \cdot \mathcal{D}_{11} + 10 \cdot \mathcal{D}^4_{10} + (12 \cdot \mathcal{D}^3_{01} + 8 \cdot \mathcal{D}^2_{01}) + 4 \cdot \mathcal{D}^{24}_{00};$$

\textbf{Subcase 4.} $deg(x'_5) = deg(x'_4) = 2$ and $f'_3$ is adjacent to end faces $f'_2$ and $f''_2$. Since $G_L$ has no vertices of degree 3 in $G$, only two configurations are possible:

\textbf{Subcase 4a.} $deg(c_2) = 6$.
$$i(G) = 161 \cdot \mathcal{I}_{00} + 98 \cdot \mathcal{I}_{01} + 81 \cdot \mathcal{I}_{10} > 32 \cdot (5 \cdot \mathcal{I}_{00} +  2 \cdot \mathcal{I}_{01} +  2 \cdot \mathcal{I}_{10}) = 32 \cdot i(G');$$
$$\partial_4(G) \leq 36 \cdot (\mathcal{D}_{11} + \mathcal{D}^3_{10} + \mathcal{D}^3_{01} + \mathcal{D}^{22}_{00}) \leq 36 \cdot (2 \cdot \mathcal{D}_{11} + \mathcal{D}^3_{10} + \mathcal{D}^3_{01} + \mathcal{D}^{22}_{00}).$$

\textbf{Subcase 4b.} $deg(c_2) = 5$.
$$i(G) = 175 \cdot \mathcal{I}_{00} + 70 \cdot \mathcal{I}_{01} + 81 \cdot \mathcal{I}_{10} > 35 \cdot (5 \cdot \mathcal{I}_{00} +  2 \cdot \mathcal{I}_{01} +  2 \cdot \mathcal{I}_{10}) = 35 \cdot i(G');$$
$$\partial_4(G) = 39 \cdot \mathcal{D}_{11} + (12 \cdot \mathcal{D}^4_{10} + 14 \cdot \mathcal{D}^3_{10}) + (16 \cdot \mathcal{D}^3_{01} + 16 \cdot \mathcal{D}^2_{01}) + 8 \cdot \mathcal{D}^{33}_{00} + 4 \cdot \mathcal{D}^{23}_{00};$$
$$\partial_4(G') = 2 \cdot \mathcal{D}_{11} + \mathcal{D}^3_{10} + \mathcal{D}^3_{01} + \mathcal{D}^{22}_{00}.$$

\textbf{Subcase 5.} $deg(x'_5) = deg(x'_4) =  deg(x'_3) = 2$ and $f'_2$ is adjacent to end faces $f'_1$ and $f''_1$. Again, since $G_L$ has no vertices of degree 3 in $G$, only two configurations are possible:

\textbf{Subcase 5a.} $deg(c_2) = 6$.
$$i(G) = 245 \cdot \mathcal{I}_{00} + 126 \cdot \mathcal{I}_{01} + 126 \cdot \mathcal{I}_{10} > 49 \cdot (5 \cdot \mathcal{I}_{00} +  2 \cdot \mathcal{I}_{01} +  2 \cdot \mathcal{I}_{10}) = 49 \cdot i(G');$$
$$\partial_4(G) \leq 52 \cdot (\mathcal{D}_{11} + \mathcal{D}^3_{10} + \mathcal{D}^3_{01} + \mathcal{D}^{22}_{00}) \leq 52 \cdot (2 \cdot \mathcal{D}_{11} + \mathcal{D}^3_{10} + \mathcal{D}^3_{01} + \mathcal{D}^{22}_{00}).$$

\textbf{Subcase 5b.} $deg(c_2) = 5$.
$$i(G) = 259 \cdot \mathcal{I}_{00} + 98 \cdot \mathcal{I}_{01} + 126 \cdot \mathcal{I}_{10} > 51 \cdot (5 \cdot \mathcal{I}_{00} +  2 \cdot \mathcal{I}_{01} +  2 \cdot \mathcal{I}_{10}) = 52 \cdot i(G');$$
$$\partial_4(G) = 52 \cdot \mathcal{D}_{11} + (18 \cdot \mathcal{D}^4_{10} + 18 \cdot \mathcal{D}^3_{10} + 3 \cdot \mathcal{D}^2_{10}) + (24 \cdot \mathcal{D}^3_{01} + 16 \cdot \mathcal{D}^2_{01}) + 12 \cdot \mathcal{D}^{23}_{00} + 4 \cdot \mathcal{D}^{22}_{00}.$$

\textbf{Case 2.} $f'_5$ contains the edge $d_1d_2$. Denote by $d_3$ the third vertex of $f'_5$. In each of the following subcases we consider the partition $(G_L,G_R,d_1d_2)$ of the graph $G$ and the partition $(G'_L,G_R,d_1d_2)$ of the graph $G'$, where $G'$ is a spanning subgraph of $G$ with the vertex set $(V(G) \setminus V(G_L)) \cup \{c_1,c_2,d_1,d_2,d_3\}$. Clearly, we have 
$$\partial_4(G') = 2 \cdot \mathcal{D}_{11} + \mathcal{D}^3_{10} + \mathcal{D}^3_{01} + \mathcal{D}^{22}_{00}.$$

It is not hard to check using Lemma 1, that in all  subcases below we have $i(G) > \partial_4(G)$.

\textbf{Subcase 1.} $deg(x'_5) = 1$.
$$i(G) = 37 \cdot \mathcal{I}_{00} + 14 \cdot \mathcal{I}_{01} + 14 \cdot \mathcal{I}_{10} > 7 \cdot (5 \cdot \mathcal{I}_{00} +  2 \cdot \mathcal{I}_{01} +  2 \cdot \mathcal{I}_{10}) = 7 \cdot i(G');$$
$$\partial_4(G) = 10 \cdot \mathcal{D}_{11} + 6 \cdot \mathcal{D}^3_{10} + (3 \cdot \mathcal{D}^4_{01} + \mathcal{D}^3_{01}) + (3 \cdot \mathcal{D}^{32}_{00} + \mathcal{D}^{22}_{00}).$$

\textbf{Subcase 2.} $deg(x'_5) = 3$, $f'_5$ is adjacent to end faces $f'_4$ and $f''_4$.
$$i(G) = 70 \cdot \mathcal{I}_{00} + 28 \cdot \mathcal{I}_{01} + 45 \cdot \mathcal{I}_{10} > 14 \cdot (5 \cdot \mathcal{I}_{00} +  2 \cdot \mathcal{I}_{01} +  2 \cdot \mathcal{I}_{10}) = 7 \cdot i(G');$$
$$\partial_4(G) = 16 \cdot \mathcal{D}_{11} + (6 \cdot \mathcal{D}^4_{10} + 4 \cdot \mathcal{D}^3_{10}) + (6 \cdot \mathcal{D}^4_{01} + 5 \cdot \mathcal{D}^3_{01} + \mathcal{D}^2_{01}) + (3 \cdot \mathcal{D}^{23}_{00} +  \mathcal{D}^{12}_{00}).$$

\textbf{Subcase 3.} $deg(x'_5) = 2$ and $f'_4$ is adjacent to end faces $f'_3$ and $f''_3$. Two configurations are possible:

\textbf{Subcase 3a.} $deg(d_1) = 6$, $d_1$ belongs to $f'_4$.
$$i(G) = 116 \cdot \mathcal{I}_{00} + 70 \cdot \mathcal{I}_{01} + 49 \cdot \mathcal{I}_{10} > 23 \cdot (5 \cdot \mathcal{I}_{00} +  2 \cdot \mathcal{I}_{01} +  2 \cdot \mathcal{I}_{10}) = 23 \cdot i(G');$$
$$\partial_4(G) = 22 \cdot \mathcal{D}_{11} + 12 \cdot \mathcal{D}^3_{10} + 4 \cdot \mathcal{D}^3_{10} + (9 \cdot \mathcal{D}^4_{01} + 6 \cdot \mathcal{D}^3_{01} + \mathcal{D}^2_{01}) + 4 \cdot (6 \cdot \mathcal{D}^{32}_{00} + 2 \cdot \mathcal{D}^{22}_{00}).$$


\textbf{Subcase 3b.} $deg(d_1) = 4$, $d_1$ does not belong to $f'_4$.
$$i(G) = 98 \cdot \mathcal{I}_{00} + 28 \cdot \mathcal{I}_{01} + 49 \cdot \mathcal{I}_{10} > 18 \cdot (5 \cdot \mathcal{I}_{00} +  2 \cdot \mathcal{I}_{01} +  2 \cdot \mathcal{I}_{10}) = 7 \cdot i(G');$$
$$\partial_4(G) = 26 \cdot \mathcal{D}_{11} + 6 \cdot \mathcal{D}^4_{10} + (3 \cdot \mathcal{D}^3_{01} + 2 \cdot \mathcal{D}^2_{01}) + 3 \cdot \mathcal{D}^{34}_{00} + 2 \cdot \mathcal{D}^{24}_{00}.$$

\textbf{Subcase 4.} $deg(x'_5) = deg(x'_4) =  2$ and $f'_3$ is adjacent to two end faces $f'_2$ and $f''_2$.

\textbf{Subcase 4a.} $deg(d_1) = 5$.
$$i(G) = 171 \cdot \mathcal{I}_{00} + 98 \cdot \mathcal{I}_{01} + 63 \cdot \mathcal{I}_{10} > 33 \cdot (5 \cdot \mathcal{I}_{00} +  2 \cdot \mathcal{I}_{01} +  2 \cdot \mathcal{I}_{10}) = 25 \cdot i(G');$$
$$\partial_4(G) = 36 \cdot \mathcal{D}_{11} + (18 \cdot \mathcal{D}^3_{10} + 14 \cdot \mathcal{D}^2_{10}) + (12 \cdot \mathcal{D}^4_{01} + 10 \cdot \mathcal{D}^3_{01} + 2 \cdot \mathcal{D}^2_{01}) + 9 \cdot \mathcal{D}^{32}_{00} + 3 \cdot \mathcal{D}^{22}_{00}.$$

\textbf{Subcase 4b.} $deg(d_1) = 4$.
$$i(G) = 189 \cdot \mathcal{I}_{00} + 70 \cdot \mathcal{I}_{01} + 63 \cdot \mathcal{I}_{10} > 35 \cdot (5 \cdot \mathcal{I}_{00} +  2 \cdot \mathcal{I}_{01} +  2 \cdot \mathcal{I}_{10}) = 35 \cdot i(G');$$
$$\partial_4(G) = 36 \cdot \mathcal{D}_{11} + (12 \cdot \mathcal{D}^4_{10} + 6 \cdot \mathcal{D}^3_{10}) + (12 \cdot \mathcal{D}^4_{01} + 4 \cdot \mathcal{D}^3_{01}) + 12 \cdot \mathcal{D}^{33}_{00}.$$

\textbf{Subcase 5.} $deg(x'_5) = deg(x'_4) =  deg(x'_3) = 2$ and $x'_2$ is adjacent to two end faces $x'_1$ and $x''_1$.

\textbf{Subcase 5a.} $deg(d_1) = 5$.
$$i(G) = 259 \cdot \mathcal{I}_{00} + 126 \cdot \mathcal{I}_{01} + 98 \cdot \mathcal{I}_{10} > 51 \cdot (5 \cdot \mathcal{I}_{00} +  2 \cdot \mathcal{I}_{01} +  2 \cdot \mathcal{I}_{10}) = 51 \cdot i(G');$$
$$\partial_4(G) = 52 \cdot \mathcal{D}_{11} + (24 \cdot \mathcal{D}^3_{10} + 16 \cdot \mathcal{D}^2_{10}) + (18 \cdot \mathcal{D}^4_{01} + 5 \cdot \mathcal{D}^3_{01} + 4 \cdot \mathcal{D}^2_{01}) + 12 \cdot \mathcal{D}^{32}_{00} + 4 \cdot \mathcal{D}^{22}_{00}.$$

\textbf{Subcase 5b.} $deg(d_1) = 4$.
$$i(G) = 277 \cdot \mathcal{I}_{00} + 98 \cdot \mathcal{I}_{01} + 98 \cdot \mathcal{I}_{10} > 52 \cdot (5 \cdot \mathcal{I}_{00} +  2 \cdot \mathcal{I}_{01} +  2 \cdot \mathcal{I}_{10}) = 52 \cdot i(G');$$
$$\partial_4(G) = 52 \cdot \mathcal{D}_{11} + (18 \cdot \mathcal{D}^4_{10} + 6 \cdot \mathcal{D}^3_{10}) + (18 \cdot \mathcal{D}^4_{01} + 6 \cdot \mathcal{D}^3_{01}) + 16 \cdot \mathcal{D}^{33}_{00}.$$

\end{proof}

Lemmas~\ref{lem2}--\ref{lem9} imply the main result of this paper.
\setcounter{theorem}{0}
\begin{theorem}\label{thm1}
For every outerplanar graph $G$ we have $i(G) > \partial_4(G)$.
\end{theorem}

\section{Concluding remarks}

It seems that the following generalization of Theorem 1 is true.

\begin{conjecture}
For every graph $G$ with the average vertex degree at most $k \geq 1$ the inequality $i(G) \geq \partial_k(G)$ holds. Moreover, equality occurs if and only if $G$ is $k$-regular.
\end{conjecture}

Although we are unable to prove this statement even for $k = 4$, it is easy to obtain a similar result for the class of trees.

\begin{theorem}\label{thm2}
For every tree $T$ we have $i(T) > \partial_2(T)$.
\end{theorem}

\begin{proof}
Clearly, for every tree with at most 3 vertices the inequality holds. Let $T$ be a $n$-vertex tree such that $i(T) \leq \partial_2(T)$ and for every tree $T'$ such that $|V(T')| < |V(T)|$ we have $i(T') > \partial_2(T')$. Consider a diametral path $X = x_1x_2x_3\ldots x_k$ in $T$. If $k \leq 3$ then $\partial_2(T) \leq 2$ and $i(G) \geq 5$, thus we assume that $k \geq 4$. Let $T_2$ (resp. $T_3$) be the maximal by inclusion subtree of $T$ such that $x_2,x_3 \in V(T_2)$ and $deg(x_2) = 1$ (resp. $x_3,x_4 \in V(T_3)$ and $deg(x_3) = 1$). Since all neighbors of $x_2$, except possibly $x_3$, belong to every 2-DS of $T$, we have $$\partial_2(T) = \partial_2(T,x^+_2) + \partial_2(T,x^-_2) \leq \partial_2(T_2) + \partial_2(T_3).$$ On the other hand, $$i(T) = i(T,x^-_1) + i(T,x^+_1) \geq i(T_2) + i(T_3).$$
Since $i(T_2) > \partial_2(T_2)$ and $i(T_3) > \partial_2(T_3)$, we have $i(T) > \partial_2(T)$.

\end{proof}

\section{Acknowledgements}

This work was performed at the Saint Petersburg Leonhard Euler International Mathematical
Institute and supported by the Ministry of Science and Higher Education of the Russian
Federation (agreement no. 075–15–2022–287).

\end{document}